\documentclass[a4paper,11pt]{article}
\usepackage{authblk}
\usepackage[english]{babel} 
\usepackage[centertags,intlimits]{amsmath}
\usepackage{amsfonts}
\usepackage{amsthm}
\usepackage{amssymb}
\usepackage{pgf, tikz, pgfplots, float, graphicx, tikz-cd}
\usetikzlibrary{decorations.pathreplacing,calligraphy, arrows}

\newtheorem{theorem}{Theorem}[section]
\newtheorem{proposition}[theorem]%
{Proposition}

\newtheorem{lemma}[theorem]%
{Lemma}
\newtheorem{corollary}[theorem]%
{Corollary}

\newtheorem{example}{Example}[section]

\begin{document}

\title{On Proximity and other Distance Parameters in Planar Graphs}

\author[1]{Peter Dankelmann\thanks{Financial support by the South African National Research Foundation is greatly acknowledged.}}
\author[1,2]{Sonwabile Mafunda}
\author[1]{Sufiyan Mallu\thanks{The results of this paper form part of third author's PhD thesis.}
\thanks{Financial support by the South African National Research Foundation is greatly acknowledged.}}
\affil[1]{University of Johannesburg\\
South Africa}
\affil[2]{Soka University of America\\
USA}
%\affil[3]{National Institute for Theoretical and Computational Sciences, South Africa\\}% (NITheCS)\\South Africa\\}

%\author{ , 
%,
%    \\
%University of Johannesburg}
%\date{October 2023}

\maketitle

\begin{abstract}
Let $G$ be a connected graph. The average distance of a vertex $v$ of $G$ is the arithmetic mean of the distances 
from $v$ to all other vertices of $G$. The proximity and remoteness of $G$ are defined as the 
minimum and maximum, respectively, of the average distances of the vertices of $G$.

It was shown by Aouchiche and Hansen [Proximity and remoteness in graphs: bounds and conjectures, Networks 58 no.\ 2 (2011)]
that for a connected graph of order $n$, the difference between remoteness and proximity and the difference
between radius and proximity are bounded from above by about $\frac{n}{4}$, and the difference between 
diameter and proximity is bounded from above by about $\frac{3}{4}n$.  

In this paper, we show that all three bounds can be improved significantly for maximal planar graphs, and for graphs
of given connectivity. 

We show that in maximal planar graphs the above bound on the difference between radius and proximity 
can be improved to about $\frac{1}{12}n$, 
and further to about $\frac{1}{16}n$ and $\frac{1}{20}n$ if the graphs is, in addition, $4$-connected or
$5$-connected, respectively. Similar improvements are shown for quadrangulations, and for maximal outerplanar graphs. 
We further show that the above bound on the difference between remoteness and proximity can be improved 
to about $\frac{1}{4\kappa}n$ if $G$ is $\kappa$-connected. 
Finally, we improve the bound on the difference between diameter and proximity to about $\frac{3}{4\kappa}n$ if $G$ is $\kappa$-connected. 
We present graphs that demonstrate that our bounds are either sharp, or sharp apart from an additive constant, 
even if restricted to planar graphs.  
\end{abstract}

MSC-class: 05C12

\section{Introduction}
In this paper we consider finite, connected graphs with no loops or multiple edges. 
In a connected graph $G$ of order $n \geq 2$ with vertex set $V$,
the average distance $\bar{\sigma}_G(v)$ of a vertex is defined to be the arithmetic mean 
of the distances from $v$ to all other vertices of $G$, i.e., $\overline{\sigma}_G(v) = \frac{1}{n-1} \sum_{w \in V} d_G(v,w)$, 
where the distance $d_G(v,w)$ denotes the usual shortest path distance. The proximity $\pi(G)$ 
and remoteness $\rho(G)$ of $G$ are defined as the minimum and maximum, respectively, 
of the average distances of the vertices of $G$, i.e.,  
\[\pi(G)={\rm min}_{v\in V}\bar{\sigma}_G(v), \quad \rho(G)={\rm max}_{v\in V}\bar{\sigma}_G(v). \] 
The terms proximity and remoteness were first used in a paper on automated comparison of graph invariants 
\cite{AouCapHan2007}, and are in wide use nowadays. However, the proximity of graphs and closely related concepts 
had been studied before under different names.  
Zelinka  [29] studied the {\em vertex deviation}, defined as $\frac{\sigma_G(v)}{n}$, 
where $\sigma_G(v)$ denotes the sum of the distances between $v$ and all other vertices,
and $n$ is the number of vertices.
Also the name {\em minimum status}, defined as $\min_{v \in V(G)} \sigma(v)$ has been used (see, for example, \cite{LiaZhoGuo2021}).

Bounds on proximity and remoteness in terms of order only were given by Zelinka \cite{Zel1968} 
and later, independently, by Auochiche and Hansen \cite{AouHan2011}. 

\begin{theorem}\label{Zel1968AouHan2011}
{\rm (Zelinka \cite{Zel1968}, Aouchiche, Hansen \cite{AouHan2011})} \\
(a) Let $G$ be a connected graph of order $n\geq 2$. Then
\[ \pi(G) \leq \left\{ \begin{array}{cc}
  \frac{n+1}{4} & \textrm{if $n$ is odd,} \\ 
  \frac{n+1}{4} + \frac{1}{4(n-1)} & \textrm{if $n$ is even,} 
       \end{array} \right. \]
with equality if and only if $G$ is either a path or a cycle. \\
(b) Let $G$ be a connected graph of order $n\geq 2$. Then        
\[ \rho(G) \leq \frac{n}{2}, \]
with equality if and only if $G$ is a path.
\end{theorem}

Relations between proximity and other graph parameters have been studied, 
for example girth \cite{AouHan2017},  
maximum degree see \cite{DanMafMal2022, LinTsaShaZha2011, Mal2022} 
clique number see \cite{HuaDas2014},
matching number and domination number \cite{LiaZhoGuo2021, PeiPanWanTia2021},
radius \cite{RisBur2014}
distance eigenvalues \cite{LinDasWu2016}. 
The proximity of trees with no vertex of degree $2$ was investigated in \cite{CheLinZho2022}.
For a recent survey on proximity and remoteness in graphs see \cite{AouRat2024}. 
Recently, directed versions of proximity and remoteness were introduced by Ai, Gerke, Gutin 
and Mafunda \cite{AiGerGutMaf2021}, who proved the following bound. 

\begin{theorem}\label{AiGerGutMaf2021}
{\rm (Ai, Gerke, Gutin, Mafunda \cite{AiGerGutMaf2021})}\\
Let $D$ be a strong digraph of order $n\geq 3$. Then 
\[1 \leq \pi(D)\leq \frac{n}{2},\]
\[1 \leq \rho(D)\leq \frac{n}{2},\]
and both bounds are sharp.
\end{theorem}

It was shown in  \cite{Dan2015} that the bounds in Theorem \ref{Zel1968AouHan2011} can be improved by a factor of about
$\frac{3}{\delta+1}$ for graphs of minimum degree $\delta$. In \cite{DanJonMaf2021}, these bounds were 
improved further for graphs not containing cycles of length $3$ or $4$. 
For maximal planar graphs and related types of graphs, the following bounds on proximity and 
remoteness were given in \cite{CzaDanOlsSze2022} and \cite{CzaDanOlsSze2021}, respectively. 

\begin{theorem}
{\rm (Czabarka, Dankelmann, Olsen, Sz\'{e}kely \cite{CzaDanOlsSze2021})} \\
(a) If $G$ is a maximal planar graph of order $n$, then
\[ \rho(G) \leq \frac{n+2}{6} + \varepsilon_n,  \quad \pi(G) \leq \frac{n+19}{12} + \frac{25}{3(n-1)},  \]
where $\varepsilon_n = 0$ if $n \equiv 1 \pmod{3}$, and $\varepsilon_n = \frac{1}{3(n-1)}$ otherwise. \\ 
(b) If $G$ is a $4$-connected maximal planar graph of order $n$, then
\[ \rho(G) \leq \frac{n+3}{8} + \varepsilon_n, \quad \pi(G) \leq \frac{n+35}{16} + \frac{91}{4(n-1)}, \]
where $\varepsilon_n = 0$ if $n \equiv 1 \pmod{4}$, $\varepsilon_n = \frac{1}{2(n-1)}$ if $n \equiv 3 \pmod{4}$, 
and $\varepsilon_n = \frac{3}{8(n-1)}$ otherwise. \\
(c) If $G$ is a $5$-connected maximal planar graph of order $n$, then
\[ \rho(G) \leq \frac{n+4}{10} + \varepsilon_n, \quad \pi(G) \leq \frac{n+57}{20} + \frac{393}{10(n-1)}, \]
where $\varepsilon_n = -\frac{3}{5(n-1)}$ if $n \equiv 0 \pmod{5}$, $\varepsilon_n = -\frac{1}{n-1}$ if $n \equiv 1 \pmod{5}$, 
$\varepsilon_n = \frac{2}{5(n-1)}$ if $n \equiv 2 \pmod{5}$, 
and $\varepsilon_n = -\frac{2}{5(n-1)}$ otherwise. \\
In (a)-(c), the bounds on remoteness are sharp, and the bounds on proximity are sharp apart from an additive constant. 
\end{theorem}

Bounds on proximity and remoteness in outerplanar graphs in terms of order and maximum face length
were given in \cite{DanMafMal-manu}.

It is natural to ask how large the differences between proximity, remoteness,
and other distance parameters can be. 
This question was answered for the difference between proximity and remoteness by 
Aouchiche and Hansen \cite{AouHan2011}, who obtained the following sharp bound.

\begin{theorem} \label{theo:rho-vs-pi-given-n}
{\rm (Aouchiche, Hansen \cite{AouHan2011})} \\
Let $ G $ be a connected graph of order $ n \geq 3 $. Then
\[\rho(G) - \pi(G) \leq \left\{
 	\begin{array}{ll}
 	\dfrac{n - 1}{4}\;\quad\qquad\;\;\quad\;\;\;\;\;\mbox{if n is odd,}\\
 	\\
 	\dfrac{n - 1}{4} - \dfrac{1}{4(n - 1)}\;\quad\mbox{if n is even.}\\
 	\end{array}
 	\right.\]
Equality holds if and only if \( G \) is a graph constructed by joining an endpoint of a path 
$P_{\lceil \frac{n}{2} \rceil}$ with any vertex of a connected graph \( H \) on 
\( \lfloor \frac{n}{2} \rfloor + 1 \) vertices.
\end{theorem}

Aouchiche and Hansen \cite{AouHan2011} determined also the maximum differences  between diameter 
 and proximity, and between radius and proximity in a connected graph.

\begin{theorem} \label{theo:diameter-vs-pi-given-n}
{\rm (Aouchiche, Hansen \cite{AouHan2011})} \\
Let $G$ be a connected graph of order $n \geq 3$. Then 
\[ {\rm diam}(G) - \pi(G) \leq \left\{ \begin{array}{cc} 
    \frac{3n-5}{4} & \textrm{if $n$ is odd,} \\ 
    \frac{3n-5}{4} - \frac{1}{4n-4} & \textrm{if $n$ is even.} 
        \end{array} \right. \]
Equality holds if and only if $G$ is a path.             
\end{theorem}

\begin{theorem}  \label{theo:radius-vs-pi-given-n}
{\rm (Aouchiche, Hansen \cite{AouHan2011})} \\
Let $G$ be a connected graph of order $n \geq 3$. Then 
\[ {\rm rad}(G) - \pi(G) \leq \left\{ \begin{array}{cc} 
    \frac{n-1}{4} - \frac{1}{n-1} & \textrm{if $n$ is odd,} \\ 
    \frac{n-1}{4} - \frac{1}{4n-4} & \textrm{if $n$ is even.} 
        \end{array} \right. \]
and this bound is sharp.              
\end{theorem}

The extremal graphs maximising the differences $\rho(G) - \pi(G)$ in 
Theorem \ref{theo:rho-vs-pi-given-n} and ${\rm diam}(G)-\pi(G)$ in 
Theorem \ref{theo:diameter-vs-pi-given-n}
contain many vertices of 
degree $2$. Hence it is natural to expect that this bound can be improved for graphs in which
the vertex degrees are larger. For graphs of given minimum degree $\delta$, the bounds in 
Theorems \ref{theo:rho-vs-pi-given-n} to \ref{theo:radius-vs-pi-given-n} were improved by a 
factor of about $\frac{3}{\delta+1}$ 
in \cite{Dan2016}, and for further improved bounds see \cite{DanMaf2022}.

Also other differences between distance parameters were considered in the literature, for 
example between radius and remoteness \cite{HuaCheDas2015}, between average eccentricity
(defined as the arithmetic mean of the eccentricities of all vertices) and proximity
\cite{MaWuZha2012}, between remoteness and radius \cite{WuZha2014}, between  
remoteness and average distance (defined as the arithmetic mean of the distances
between all pairs of distinct vertices) \cite{WuZha2014}
and between proximity and average distance \cite{Sed2013}.

The aim of this paper is to improve the bounds on the differences between proximity
and the three distance parameters remoteness, diameter and radius in 
Theorems \ref{theo:rho-vs-pi-given-n}, \ref{theo:diameter-vs-pi-given-n} and
\ref{theo:radius-vs-pi-given-n} 
for maximal planar and maximal outerplanar graphs and for graphs of given connectivity.
We show that for maximal planar graphs the bound in Theorem \ref{theo:radius-vs-pi-given-n} 
can be improved to about $\frac{1}{12}n$, with further improvements to 
about $\frac{1}{16}n$ and $\frac{1}{20}n$ for $4$-connected and $5$-connected maximal
planar graphs, respectively. We also improve the bounds in 
Theorems \ref{theo:rho-vs-pi-given-n} and \ref{theo:diameter-vs-pi-given-n} by a factor 
of $\frac{1}{\kappa}$ for $\kappa$-connected graphs.  

This paper is organised as follows. 
In Section \ref{section:terminology}, we define the main terms and notation used in this paper.
Some results on planar and outerplanar graphs that are used in the sections that follow are presented
in Section \ref{section:preliminaries-on-(outer-)planar}.
In Section \ref{section:radius-vs-proximity} we prove bounds on the difference between radius and 
proximity in maximal planar graphs and quadrangulations of given connectivity, and for maximal outerplanar graphs. 
The difference between remoteness and proximity in graphs of given connectivity is considered in 
Section \ref{section:remoteness-vs-proximity}, and the last section, Section \ref{section:diameter-vs-proximity},
presents bounds on the difference between diameter and proximity in graphs of given connectivity.

\section{Terminology and notation}
\label{section:terminology}

We use the following notation. 
Let $G$ be a connected graph and let $u$ and $v$ be vertices of $G$.
We denote by $V(G)$ and $E(G)$ the {\em vertex set} and {\em edge set}, respectively.  
The {\em order} $n$ of a graph is  the cardinality of the vertex set. 

The {\em distance} $d_G(u,v)$ between $u$ and $v$ is the minimum number of edges
on a shortest path from $u$ to $v$. 
The {\em eccentricity} ${\rm ecc}_G(v)$ of a vertex $v$ in a graph $G$ is the distance 
from $v$ to a vertex farthest from $v$. The {\em diameter} is the largest of all
eccentricities of vertices of $G$, and the {\em radius} is the smallest of all
eccentricites of vertices of $G$, they are  denoted by 
${\rm diam}(G)$ and ${\rm rad}(G)$, respectively.
For a vertex $v$ and a set $X$ of vertices, we denote the sum $\sum_{x \in X} d_{G}(v, x)$ by $\sigma(v|X)$.
The {\em total distance} of a vertex $v$ is defined by 
$\sigma(v) = \sum_{w \in V}d(v,w)$.

Let $i\in \mathbb{Z}$. Then $N_i(v)$, $N_{\leq i}(v)$ and $N_{\geq i}(v)$ denote the set of all 
vertices of $G$ at distance exactly, at most and at least $i$, respectively, from $v$.  
We write $n_i(v)$ for $|N_i(v)|$. Clearly, $n_i(v) >0$ if and only if $0\leq i \leq {\rm ecc}_G(v)$. 
The {\em neighbourhood} $N(v)$ of $v$ is the set of vertices adjacent to $v$, i.e., $N_1(v)$.

We denote the {\em complete graph} of order $n$ by $K_n$, the {\em cycle} of order $n$ by
$C_n$,  and the {\em edgeless graph} or order $n$ by $\overline{K_n}$.
If $G_1, G_2,\ldots,G_k$ are graphs, then the sequential sum $G_1 + G_2+ \ldots + G_k$ is
the graph obtained from the disjoint union of $G_1, G_2,\ldots, G_k$ by adding an edge between each 
vertex of $G_i$ and each vertex of $G_{i+1}$ for $i=1,2,\ldots,k-1$. 
By $[G_1+G_2+\ldots+G_k]^{k}$ we mean $k$ repetitions of the pattern $G_1+G_2+\ldots+G_k$.

A graph $G$ is {\em planar} if can be embedded in the plane with no edges crossing. 
A planar graph is {\em outerplanar} if it can be embedded in the plane so that no edges 
cross and every vertex is on the boundary of the outer face. 
A {\em maximal} (outer-)planar graph is an (outer-)planar graph which is no longer (outer-)planar 
after addition of any edge. A {\em quadrangulation} (sometimes called simple quadrangulation) is
a graph that is planar and bipartite, but after adding any edge it is no longer planar and bipartite.
Since every $3$-connected planar graph, every quadrangulation and every $2$-connected outerplanar
graph has a unique embedding in
the plane, we often assume that the graph has been embedded in the plane, and we use
the terms (outer-)planar and (outer-)plane interchangeably for such graphs.

\section{Preliminary results on (outer)planar graphs}
\label{section:preliminaries-on-(outer-)planar}

In this section we present some results which will be used in the sections that follow.

Let vertex $v$ be fixed, then we say a vertex $u \in N_i(v)$ is \textbf{active} if $u$ has a neighbour in $N_{i+1}(v)$. The set of active vertices within $N_i(v)$ is denoted by $A_i(v)$. For \( 1 \leq i \leq {\rm ecc}(v) - 1 \), we define \( \hat{H_i} \) to be the graph with vertex set \( A_i \), where two vertices are adjacent if and only if they share a face in \( G \).

\begin{lemma}(\cite{AliDanMuk2012})\label{la:active-vertices-planar}
Let \( G \) be a planar graph, let $v \in V(G)$ and  $1 \leq i \leq {\rm ecc}(v) - 1$. 
Let \( A_i(v) \) and $\hat{H_i}$ be as defined above for:
\begin{itemize}
    \item[(a)] If \( G \) is 3-connected and \( u \) is a vertex in \( \hat{H_i} \), then \( u \) 
    has two distinct neighbours \( v, w \in A_i(v) - \{u\} \) in \( \hat{H_i} \).
    \item[(b)] If \( G \) is 4-connected and \( u, w, x \) are three distinct vertices in 
    \( \hat{H_i} \), then at least one of them has a neighbour in \( A_i(v) - \{u, w, x\} \) 
    in \( \hat{H_i} \).
    \item[(c)] If \( G \) is 5-connected and \( u \) is a vertex in \( \hat{H_i} \), then \( u \) has two neighbours, \( w \) and \( x \), in \( \hat{H_i} \) such that \( w \) and \( x \) share no common neighbour in \( \hat{H_i} \) other than \( u \).
\end{itemize}
\end{lemma}

\begin{lemma}(\cite{CzaDanOlsSze2022})\label{la:active-vertices-quadrangulation}
Let \( G \) be a quadrangulation, \( v \) a vertex of \( G \), and let 
\( 1 \leq i \leq {\rm ecc}(v) - 1 \). For every active vertex \( w \in N_i(v) \), there exists 
another active vertex \( w' \in N_i(v) \) such that \( w \) and \( w' \) share a common face of 
\( G \).
\end{lemma}

We now prove a corresponding result for outerplanar graphs. Its proof uses a similar approach as Lemma \ref{la:active-vertices-planar} 
in \cite{AliDanMuk2012}. 

\begin{lemma}\label{lemm:2-connected outerplanar}
Let $G$ be a $2$-connected outerplanar graph, $v$ a vertex of $G$, and 
$i \in \{1,2,\ldots,{\rm ecc}(v)-1\}$. For every active vertex $u \in N_{i}(v)$, there exist another active vertex $u' \in N_{i}(v)$ such that $u$ and $u'$ share a face in $G$ which is distinct from 
the outer face.
 \end{lemma}
 
\begin{proof}
Let $u\in N_i(v)$ be an arbitrary active vertex. Label the neighbours of $u$ as $x_0,x_1,\ldots,x_t$ such that the edges $ux_i$ appear in, 
say, clockwise order, and edges $ux_0$ and $ux_t$ are on the Hamiltonian cycle $C$. Denote the face containing $u$, $x_j$ and $x_{j+1}$ by 
$f_j$ for $j \in \{0,1,\ldots,t-1\}$. Let $P_i$ be the $(x_i,x_{i+1})$-path of the vertices on the boundary of $f_i$ except $u$ in 
clockwise order. Let $x_p$ be a neighbour of $u$ in $N_{i-1}$ and let $x_q$ be a neighbor of $u$ in $N_{i+1}$. We may assume that $p < q$. 
Let $W$ be a $(x_p,x_q)$-walk that traverses the vertices of $P_{p}$ then $P_{p+1}, P_{p+2}, \ldots, P_{q-1}$. Let $z$ be the first vertex 
of $W$ in $N_{i+1}$ and let $u'$ be its predecessor. Then $u'$ is in $N_i$, and since it is adjacent to $z$, $u'$ is active. 
Moreover, $u'$ shares a face with $u$ since it is on $W$. The lemma follows.  
\end{proof}

Also the following result, which appears to be folklore, will be used in Section \ref{section:radius-vs-proximity}.

\begin{proposition}  \label{prop:radius}
Let $G$ be a connected graph of order $n$. Then 
\[{\rm rad}(G) \leq \lfloor \frac{n}{2} \rfloor.\] 
\end{proposition}

\section{Radius vs proximity}
\label{section:radius-vs-proximity}

In this section we prove upper bounds on the difference between radius and proximity in 
maximal planar graphs and quadrangulations of given connectivity, and in maximal outerplanar graphs.

\begin{theorem}\label{theo:rad-vs-pi-in-triangulation}
Let $G$ be a maximal planar graph of order $n \geq 4$. Then
\[{\rm rad}(G)-\pi(G) \leq \frac{n+1}{12}+\frac{4}{3} +\frac{27}{4(n-1)}.\]
\end{theorem}

\begin{proof}
Assume \( G \) is a maximal planar graph of order $n \geq 4$. 
Let $v_0$ be a median vertex of $G$ and let $R := {\rm ecc}(v_0)$, and let $r:= {\rm rad}(G)$. 
Recall that $N_i(v_0)$ is the set of vertices at distance $i$ from $v_0$, for 
$i \in \{0,1,\ldots,R\}$ and let \( A_i:=A_i(v_0) \) be the set of active vertices in 
\( N_i(v_0) \). 
 We show that 
\begin{equation} \label{eq:lb on v_0}
\sigma(v_0) \geq  3r^2-9r.
\end{equation} 
We consider two cases, depending on whether $R$ is much larger than $r$ or not. \\[1mm]
{\sc Case 1:} $R \geq \frac{1}{2}(3r-5)$. \\ 
Since $G$ is $3$-connected, we have that $|N_i(v_0)| \geq 3$ for $i \in \{1,2,\ldots,R-1\}$, and thus
\[
\sigma(v_0) = \sum_{i=1}^{R} |N_i(v_0)| \, i 
     > \; \sum_{i=1}^{R-1} 3i
     =  \frac{3}{2}(R^2-R). \]
Now the term $\frac{3}{2}(R^2-R)$ is increasing in $R$. Substituting $\frac{1}{2}(3r-5)$ for $R$ thus 
yields the inequality 
\[ \sigma(v_0)    
      \geq  3\left(\frac{9}{8}r^2-\frac{9}{2}r+\frac{35}{8}\right) 
       > 3r^2-9r,      \]
and (\ref{eq:lb on v_0}) follows in Case 1. \\[1mm]
{\sc Case 2:} $R \leq \frac{1}{2}(3r-6)$. \\
We first prove the following claim.\\[1mm]
{\bf Claim 1:} If $R-r+3 \leq i \leq 2r-R-3$, then $|A_i| \geq 6$. \\
Suppose to the contrary that there exist some $j \in \{R-r+3,\ldots,2r-R-3\}$ such that 
$|A_j| \leq 5$. 
It follows from Lemma \ref{la:active-vertices-planar} (a) that each component of the graph 
$\hat{H_j}$ has at least three vertices, so $\hat{H_j}$ is connected. 
By Proposition \ref{prop:radius} we have 
${\rm rad}(\hat{H}_j) \leq \lfloor \frac{|A_j|}{2} \rfloor \leq \lfloor \frac{5}{2}\rfloor =2$.
Let $x_j \in A_j$ be a vertex of eccentricity at most $2$ in $\hat{H}_j$.
Since two vertices of $G$ that are on the boundary of the same face are adjacent in $G$, 
we have $d_G(x_j,w_j) \leq 2$ for all $w_j \in A_j$. 
Let $P: v_0=x_0, x_1,\ldots,x_j$ be a shortest 
$(x_0, x_j)$-path in $G$. Consider $x_{R-r+3}$. We show that ${\rm ecc}(x_{R - r + 3}) < r$, 
leading to a contradiction, thereby establishing Claim 1. 

Since $G$ has radius $r$, there exists a
vertex $v$ of $G$ with $d(x_{R-r+3},v) \geq r$. Then $v \in N_{\geq j}(v_0)$ or 
$v \in N_{\leq j-1}(v_0)$. We now first consider the case $v \in N_{\geq j}(v_0)$. Let $w_j$ be the vertex of $N_j$ belonging to a shortest $(v_0, v)$-path. Then  $w_j \in A_j$, and so 
\begin{eqnarray*}
d(x_{R-r+3},v) & \leq &  d(x_{R-r+3},x_j)+ d(x_j,w_j)+ d(w_j,v) \\
 & \leq &  (j-R+r-3)+2+(R-j) \\
  & = &  r-1,
\end{eqnarray*}
a contradiction. Now consider the case $v \in N_{\leq j-1} $. Then
\begin{eqnarray*}
d(x_{R-r+3},v) & \leq &  d(x_{R-r+3},v_0)+ d(v_0,v) \\
   & \leq &  R-r+3+j-1 \\
  & \leq &  R-r+3+2r-R-4 \\
   & = &  r-1,
\end{eqnarray*}
again a contradiction, which completes the proof of Claim 1. 

Now $|N_i(v_0)| \geq 3$ for all $i\in \{1,2,\ldots,R\}$ since $G$ is $3$-connected, 
and $|N_i(v_0)| \geq |A_i(v_0)| \geq 6$ for $i=\{R-r+3, R-r+4,\ldots,2r-R-3\}$. Hence  
\begin{eqnarray*} 
\sigma(v_0)  & > & \; \sum_{i=1}^{R-1}|N_i(v_0)| \, i \\
             & \geq & 3\sum_{i=1}^{R-1}i+ 3\sum_{i=R-r+3}^{2r-R-3}i \\
     & =&\ 3\left(\frac{R^2-R}{2}+\frac{3r^2-2Rr-5r}{2}\right). 
\end{eqnarray*}
It is easy to verify that the derivative with respect to $R$ of the right hand side of 
the above inequality equals $3(R-r) - \frac{3}{2}$, which is positive for $R \geq r+1$. 
Hence the right hand side above is minimised, subject to
$R \in \mathbb{N}$ and $R \geq r$, if $R=r$. Substituting 
this value yields
\[ \sigma(v_0) \geq 3r^2 - 9r. \]
Hence (\ref{eq:lb on v_0}) holds also in Case 2.

From (\ref{eq:lb on v_0}) we obtain the following lower bound on the proximity. 
\[\pi(G)= \frac{\sigma(v_0)}{n-1} \geq \frac{1}{n-1}\left(3r^2 - 9r\right).\]
Thus, 
\[{\rm rad}(G)-\pi(G) \leq r - \frac{1}{n-1}\left(3r^2 - 9r\right). \]
Simple elementary calculus shows that for constant $n$, the right hand side of the above inequality is maximised for $r=\frac{n+8}{6}$. Substituting this value of $r$ into the above inequality yields,
after simplification, that 
\[{\rm rad}(G)-\pi(G) \leq \frac{n+17}{12} + \frac{27}{4(n-1)}, \]
hereby completing the proof of the theorem.
\end{proof}

The following example shows that the bound in Theorem \ref{theo:rad-vs-pi-in-triangulation} 
is sharp apart from an additive constant. 

\begin{example} \label{exa:triangulation-3-connected}
Let $n \geq 5$ with $n\equiv 5 \pmod{6}$ and let $k = \frac{n-2}{3}$. We define the graph $T_n$ as follows. 
Let $G_0$ be the graph consisting of the single vertex $b_0$. 
For each $i \in \{1, 2, \ldots, k\}$, let $G_i$ be a copy of the cycle $C_3$ with vertices $a_i, b_i, c_i$. 
Let $G_{k+1}$ be the graph consisting of a single vertex $b_{k+1}$. The graph $T_n$ is obtained from the 
disjoint union of the graphs $G_0, G_1, \ldots, G_k, G_{k+1}$ by adding edges as follows: 
the vertex $b_0$ is adjacent to $a_1, b_1,$ and $c_1$, 
for each $i \in \{1, 2, \ldots, k-1\}$ the edges $a_i a_{i+1}$, $b_i b_{i+1}$, $c_i c_{i+1}$ and $c_ia_{i+1}$ are added, 
for each $i \in \{1, 2, \ldots, \frac{k-1}{2}\}$ the edges $a_i b_{i+1}$ and $c_i b_{i+1}$ are added, 
for each $i \in \{\frac{k+1}{2}, \frac{k+1}{2}+1, \ldots, k-1\}$ the edges $b_i a_{i+1}$, and $b_i c_{i+1}$ are added, 
and finally, the edges $b_0a_1, b_0b_1, b_0c_1$ and $a_k b_{k+1}, b_k b_{k+1}, c_k b_{k+1}$ are added. 
(See Figure \ref{fig:triangulation-3-connected} for a sketch.)

Clearly, $T_n$ is maximal planar graph of order $n=3k+2$. It is easy to verify that 
${\rm rad}(T_n)= \frac{n+1}{6}$, 
$\pi(T_n) = \overline{\sigma}(b_{(k+1)/2}) = \frac{n+1}{12} + \frac{2}{n-1}$.
Hence ${\rm rad}(T_n) - \pi(T_n) = \frac{n+1}{12} - \frac{2}{n-1}$, which differs from the 
bound in Theorem \ref{theo:rad-vs-pi-in-triangulation} by $\frac{4}{3} + o(n)$.

For later use  we note that ${\rm diam}(T_n)=\frac{n+1}{3}$ and 
$\rho(T_n) = \overline{\sigma}(b_{0}) = \frac{n+2}{6} + \frac{1}{3(n-1)}$. 
\end{example}

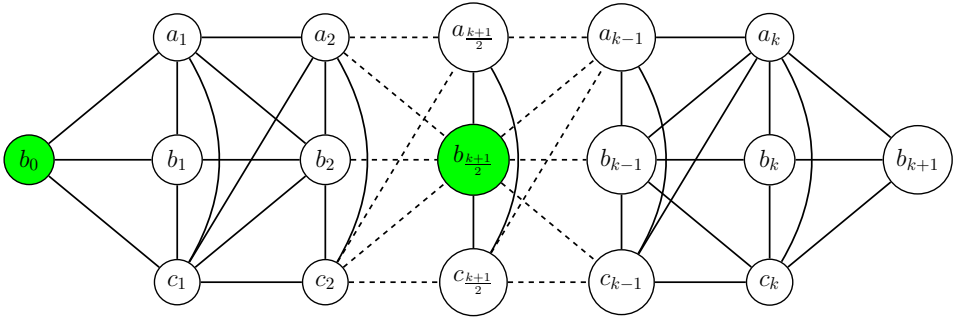
\begin{figure}[H]
\begin{center}
\resizebox{\linewidth}{!}{
\begin{tikzpicture}
[inner sep=1.5mm,
vertex/.style={circle,thick,draw,minimum size=10pt, font=\Large},
thickedge/.style={line width=1pt}]

% Nodes with labels inside
\node[vertex, fill=green] (b0) at (-12, 2.5) {$b_0$};
\node[vertex, fill=white] (b6) at (6, 2.5) {$b_{k+1}$};

\node[vertex, fill=white] (a1) at (-9, 5) {$a_1$};
\node[vertex, fill=white] (b1) at (-9, 2.5) {$b_1$};
\node[vertex, fill=white] (c1) at (-9, 0) {$c_1$};

\node[vertex, fill=white] (a2) at (-6, 5) {$a_2$};
\node[vertex, fill=white] (b2) at (-6, 2.5) {$b_2$};
\node[vertex, fill=white] (c2) at (-6, 0) {$c_2$};

\node[vertex, fill=white] (a3) at (-3, 5) {$a_{\frac{k+1}{2}}$};
\node[vertex, fill=green] (b3) at (-3, 2.5) {$b_{\frac{k+1}{2}}$};
\node[vertex, fill=white] (c3) at (-3, 0) {$c_{\frac{k+1}{2}}$};

\node[vertex, fill=white] (a4) at (0, 5) {$a_{k-1}$};
\node[vertex, fill=white] (b4) at (0, 2.5) {$b_{k-1}$};
\node[vertex, fill=white] (c4) at (0, 0) {$c_{k-1}$};

\node[vertex, fill=white] (a5) at (3, 5) {$a_k$};
\node[vertex, fill=white] (b5) at (3, 2.5) {$b_k$};
\node[vertex, fill=white] (c5) at (3, 0) {$c_k$};

% P3 vertical paths
\foreach \i in {1,...,5} {
  \draw[thickedge] (a\i)--(b\i)--(c\i);
}

% Curved edges a_i -- c_i
\foreach \i in {1,...,5} {
  \draw[thickedge] (a\i) to[bend left=30] (c\i);
}

% b0 to a1, b1, c1
\draw[thickedge] (b0)--(a1);
\draw[thickedge] (b0)--(b1);
\draw[thickedge] (b0)--(c1);

% Connections between Gi and Gi+1
\foreach \i/\j in {1/2, 4/5} {
  \draw[thickedge] (a\i)--(a\j);
  \draw[thickedge] (b\i)--(b\j);
  \draw[thickedge] (c\i)--(c\j);
%  \draw[thickedge] (a\i)--(c\j);  changed by PD
  \draw[thickedge] (c\i)--(a\j);  
}

\foreach \i/\j in {2/3, 3/4} {
  \draw[thickedge,dashed] (a\i)--(a\j);
  \draw[thickedge,dashed] (b\i)--(b\j);
  \draw[thickedge,dashed] (c\i)--(c\j);
% \draw[thickedge,dashed] (a\i)--(c\j);  changed by PD
  \draw[thickedge,dashed] (c\i)--(a\j);
}

% a_i and c_i to b_{i+1} for i = 1,2
\foreach \i/\j in {1/2} {
  \draw[thickedge] (a\i)--(b\j);
  \draw[thickedge] (c\i)--(b\j);
}

\foreach \i/\j in {2/3} {
  \draw[thickedge,dashed] (a\i)--(b\j);
  \draw[thickedge,dashed] (c\i)--(b\j);
}

% b_i to a_{i+1} and c_{i+1} for i = 3,4
\foreach \i/\j in {3/4} {
  \draw[thickedge,dashed] (b\i)--(a\j);
  \draw[thickedge,dashed] (b\i)--(c\j);
}

\foreach \i/\j in {4/5} {
  \draw[thickedge] (b\i)--(a\j);
  \draw[thickedge] (b\i)--(c\j);
}

% Final connections to b6
\draw[thickedge] (a5)--(b6);
\draw[thickedge] (b5)--(b6);
\draw[thickedge] (c5)--(b6);

\end{tikzpicture}
}
\caption{The maximal planar graph $T_n$.}
\label{fig:triangulation-3-connected}
\end{center}
\end{figure}

\begin{theorem} \label{theo:rad-vs-pi-in-triangulation-4-5-connected}
Let \( G \) be a maximal planar graph of order \( n \geq 6 \). \\
(a) If $G$ is $4$-connected, then 
\[{\rm rad}(G)-\pi(G) \leq \frac{n+31}{16}+\frac{16}{n-1}.\]
(b) If $G$ is $5$-connected, then
\[{\rm rad}(G)-\pi(G) \leq \frac{n+49}{20}+\frac{125}{4(n-1)}.\]
\end{theorem}

\begin{proof}
Both proofs follow the lines of the proof of Theorem \ref{theo:rad-vs-pi-in-triangulation}, so 
we give only the main steps. 
Let $v_0$, $N_i(v_0)$, $A_i$, $R$ and $r$ 
be as in the proof of Theorem \ref{theo:rad-vs-pi-in-triangulation}.  

(a) Since $G$ is $4$-connected, we have $|N_i(v_0)| \geq 4$ for $i=1,2,\ldots,R-1$. 
Similar to \eqref{eq:lb on v_0} we prove that 
\begin{equation}  \label{eq:lower-bound-on-sigma(v_0)-in-4-connected-planar} 
\sigma(v_0) \geq 4r^2 - 16r, 
\end{equation}
by considering the two cases $R \geq \frac{1}{2}(3r-7)$ and $R \leq \frac{1}{2}(3r-8)$. 
In the proof of the latter case, we show, with the same arguments, a slightly altered 
version of Claim 1 in 
Theorem \ref{theo:rad-vs-pi-in-triangulation}, which states that $|A_i| \geq 8$ holds
for all $i$ with $R-r+4 \leq i \leq 2r-R-4$. \\
It follows from  \eqref{eq:lower-bound-on-sigma(v_0)-in-4-connected-planar} that
the difference ${\rm rad}(G) - \pi(G)$ is bounded from above by $r - \frac{1}{n-1}(4r^2 - 16r)$,
which is maximised for $r=\frac{n+15}{8}$. Substituting this value yields 
Theorem \ref{theo:rad-vs-pi-in-triangulation-4-5-connected}(a). 

(b) Since $G$ is $4$-connected, we have $|N_i(v_0)| \geq 4$ for $i=1,2,\ldots,R-1$. 
Similar to \eqref{eq:lb on v_0} we prove that 
\begin{equation}  \label{eq:lower-bound-on-sigma(v_0)-in-5-connected-planar} 
\sigma(v_0) \geq 5r^2 - 20r, 
\end{equation}
by considering the two cases $R \geq \frac{1}{2}(3r-9)$ and $R \leq \frac{1}{2}(3r-10)$. 
In the proof of the latter case, we show, with the same arguments, a slightly altered 
version of Claim 1 in 
Theorem \ref{theo:rad-vs-pi-in-triangulation}, which states that $|A_i| \geq 10$ holds
for all $i$ with $R-r+5 \leq i \leq 2r-R-5$. \\
It follows from  \eqref{eq:lower-bound-on-sigma(v_0)-in-4-connected-planar} that
the difference ${\rm rad}(G) - \pi(G)$ is bounded from above by $r - \frac{1}{n-1}(5r^2 - 20r)$,
which is maximised for $r=\frac{n+24}{10}$. Substituting this value yields 
Theorem \ref{theo:rad-vs-pi-in-triangulation-4-5-connected}(b).  
\end{proof}

Also the bounds in Theorem \ref{theo:rad-vs-pi-in-triangulation-4-5-connected}(a) and (b) 
are sharp apart from an additive constant. This is shown by constructing $4$-connected
maximal planar graphs ($5$-connected maximal planar graphs) in a way similar to 
Example \ref{exa:triangulation-3-connected}, the main difference being that the 
$C_3$ in Example \ref{exa:triangulation-3-connected} are replaced by $C_4$ ($C_5$). We omit the details. \\

In a way similar to Theorem \ref{theo:rad-vs-pi-in-triangulation} we obtain bounds on the 
difference between radius and proximity for quadrangulations. 

Applying Lemma \ref{la:active-vertices-quadrangulation} and using analogous arguments as 
Theorem \ref{theo:rad-vs-pi-in-triangulation} we derive the following bound.

  \begin{theorem}\label{theo: rad vs in quadrangulation}
(a) If $G$ is a quadrangulation of order $n$, then
  \[
{\rm rad}(G) - \pi(G) \leq \frac{n+11}{8}+\frac{9}{2(n-1)}.
\]
(b) If $G$ is a $3$-connected quadrangulation of order $n$, then 
\[{\rm rad}(G)-\pi(G) \leq \frac{n+17}{12} + \frac{27}{4(n-1)}.\]
  \end{theorem}
  
  \begin{proof}
Let $v_0$, $N_i(v_0)$, $A_i$, $R$ and $r$ 
be as in the proof of Theorem \ref{theo:rad-vs-pi-in-triangulation}. \\[1mm]
(a) We show that 
\begin{equation} \label{eq:lb on v_0 quadrangulations}
\sigma(v_0) \geq  2r^2 - 6r.
\end{equation} 
We now consider two cases depending on whether $R$ is much larger than $r$ or not. \\[1mm]
{\sc Case 1:} $R \geq \frac{1}{2}(3r-5)$. \\
Since every quadrangulation is $2$-connected, we have that $|N_i(v_0)| \geq 2$ for 
$i \in \{0,1,\ldots,R-1\}$, and thus 
\[ \sigma(v_0) = \sum_{i=1}^R |N_i(v_0)| \, i 
       > \sum_{i=1}^{R-1}2i= R^2-R. \] 
Now the term $R^2-R$ is increasing in $R$. Hence by the defining condition of this case, 
\[ \sigma(v_0) \geq  \frac{9}{4}r^2-9r+\frac{35}{4} > 2r^2 - 6r, \]
thus (\ref{eq:lb on v_0 quadrangulations}) follows in Case 1. \\[1mm]
{\sc Case 2:} $R \leq \frac{1}{2}(3r-6)$. \\
We first prove the following claim in this case.\\[1mm]
\emph{{\bf Claim 1:}} If $R-r+3 \leq i \leq 2r-R-3$, then $|A_i| \geq 4$. \\
Suppose to the contrary that there exist some $j \in \{R-r+3,\ldots,2r-R-3\}$ such that $|A_j| \leq 3$. 
Then it follows from Lemma \ref{la:active-vertices-quadrangulation} that every component of the graph 
$\hat{H}_j$ has at least two vertices, hence $\hat{H}_j$ is connected. 
By Proposition \ref{prop:radius} we have ${\rm rad}(\hat{H}_j) \leq 1$.
Hence there is a vertex $x_j \in A_j$ which shares a face with all other vertices of $A_j$, 
and thus has distance at most $2$ from all vertices in $A_j$. Let 
$P: v_0=x_0, x_1,\ldots,x_j$ be a shortest $(x_0, x_j)$-path in $G$. Consider $x_{R-r+3}$. 
The same arguments as in the proof of Claim 1, Theorem \ref{theo:rad-vs-pi-in-triangulation}  
yield the contradiction ${\rm ecc}(x_{R - r + 3}) < r$, thereby establishing Claim 1.  

By Claim 1 we have 
\[ \sigma(v_0) 
  \geq 2\left( \sum_{i=1}^{R-1}i+\sum_{i=R-r+3}^{2r-R-3}i\right) 
   = R^2-R+3r^2-2Rr-5r. \]
The remainder of the proof is almost identical to the corresponding part of the proof of 
Theorem \ref{theo:rad-vs-pi-in-triangulation}, hence we omit it. \\[1mm]
(b) Since $G$ is $3$-connected, we have $|N_i(v_0)| \geq 3$ for $i=1,2,\ldots,R-1$. 
Similar to \eqref{eq:lb on v_0} we prove that 
\begin{equation}  \label{eq:lower-bound-on-sigma(v_0)-in-quadrangulations-3-connected} 
\sigma(v_0) \geq 3r^2 - 15r, 
\end{equation}
by considering the two cases $R \geq \frac{3}{2}(r-3)$ and $R \leq \frac{3}{2}r-5$. 
In the proof of the latter case, we make use of Lemma \ref{la:active-vertices-planar}(a) 
and show, with the same arguments, a slightly altered version of Claim 1 in 
Theorem \ref{theo:rad-vs-pi-in-triangulation}, which states that $|A_i| \geq 6$ holds
for all $i$ with $R-r+5 \leq i \leq 2r-R-5$. \\
It follows from  \eqref{eq:lower-bound-on-sigma(v_0)-in-quadrangulations-3-connected} that
the difference ${\rm rad}(G) - \pi(G)$ is bounded from above by $r - \frac{1}{n-1}(3r^2 - 15r)$,
which is maximised for $r=\frac{n+14}{6}$. Substituting this value yields 
Theorem \ref{theo: rad vs in quadrangulation}(b). 
\end{proof}

The bound in Theorem \ref{theo: rad vs in quadrangulation}(a) is sharp apart from an additive 
constant. This is shown by the following example.

\begin{example}\label{exa:quadrangulation}
Given $n\in \mathbb{N}$ with $n\geq 4$, $n \equiv 0 \pmod{4}$. Let $k= \frac{n-2}{4}$. We define the graph $Q_n$ as follows. 
\[ Q_n = K_1 + [\overline{K_{2}}]^{k} + K_1.\]
See Figure \ref{fig:A 2-connected quadrangulation} for a sketch of $Q_n$. 
Clearly, $Q_n$ is a quadrangulation of order $n=2k+2$. It is easy to verify that
${\rm rad}(Q_n) = \frac{n}{4}$ and 
$\pi(Q_n) = \overline{\sigma}(b_{(k+1)/2}) = \frac{n+17}{8} + \frac{17}{8(n-1)}$.
Hence ${\rm rad}(Q_n) - \pi(Q_n) = \frac{n-17}{8} - \frac{17}{8(n-1)}$, which differs from
the bound in Theorem \ref{theo: rad vs in quadrangulation}(a) by $\frac{17}{4} +o(n)$.
For later use  we note that 
${\rm diam}(Q_n) = \frac{n}{2}$ and 
$\rho(Q_n) = \overline{\sigma}(b_0) = \frac{n+1}{4} + \frac{1}{4(n-1)}$.
\end{example}
  
  \begin{figure}[H]
\begin{center}
\resizebox{\linewidth}{!}{
\begin{tikzpicture}
[inner sep=1.5mm,
vertex/.style={circle,thick,draw,minimum size=10pt, font=\Large},
thickedge/.style={line width=1pt}]

% Nodes with labels inside
\node[vertex, fill=green] (b0) at (-10, 2) {$b_0$};
\node[vertex, fill=white] (b6) at (8, 2) {$b_{k+1}$};

\node[vertex, fill=white] (a1) at (-7, 4) {$a_1$};
\node[vertex, fill=white] (b1) at (-7, 0) {$b_1$};

\node[vertex, fill=white] (a2) at (-4, 4) {$a_2$};
\node[vertex, fill=white] (b2) at (-4, 0) {$b_2$};

\node[vertex, fill=white] (a3) at (-1, 4) {$a_{\frac{k+1}{2}}$};
\node[vertex, fill=green] (b3) at (-1, 0) {$b_{\frac{k+1}{2}}$};

\node[vertex, fill=white] (a4) at (2, 4) {$a_{k-1}$};
\node[vertex, fill=white] (b4) at (2, 0) {$b_{k-1}$};

\node[vertex, fill=white] (a5) at (5, 4) {$a_k$};
\node[vertex, fill=white] (b5) at (5, 0) {$b_k$};

% Edge b0 to a1 and b1
\draw[thickedge] (b0) -- (a1);
\draw[thickedge] (b0) -- (b1);

% Edges a_i to a_{i+1} and b_i to b_{i+1}
\foreach \i/\j in {1/2, 4/5} {
  \draw[thickedge] (a\i) -- (a\j);
  \draw[thickedge] (b\i) -- (b\j);
}

\foreach \i/\j in {2/3, 3/4} {
  \draw[thickedge,dashed] (a\i) -- (a\j);
  \draw[thickedge,dashed] (b\i) -- (b\j);
}

% Edges a_i b_{i+1} for i=1 to (k-1)/2 = 2 (solid lines)
\foreach \i/\j in {1/2} {
  \draw[thickedge] (a\i) -- (b\j);
  \draw[thickedge] (b\i) -- (a\j);
}

\foreach \i/\j in {2/3} {
  \draw[thickedge,dashed] (a\i) -- (b\j);
  \draw[thickedge,dashed] (b\i) -- (a\j);
}

% Edges b_i a_{i+1} for i=(k+1)/2 to k-1 = 3 to 4 (dashed lines)
\foreach \i/\j in {3/4} {
  \draw[thickedge, dashed] (b\i) -- (a\j);
  \draw[thickedge, dashed] (a\i) -- (b\j);
}

\foreach \i/\j in {4/5} {
  \draw[thickedge] (b\i) -- (a\j);
  \draw[thickedge] (a\i) -- (b\j);
}

% Final connections to b_{k+1}
\draw[thickedge] (a5) -- (b6);
\draw[thickedge] (b5) -- (b6);

\end{tikzpicture}
}
\caption{The quadrangulation $Q_n$.}
\label{fig:A 2-connected quadrangulation}
\end{center}
\end{figure}
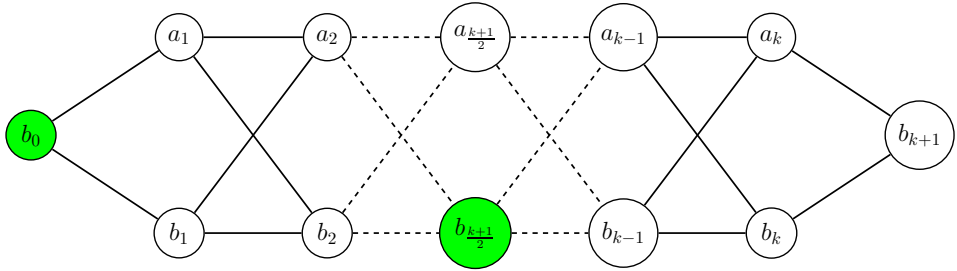

Also the bounds in Theorem \ref{theo: rad vs in quadrangulation}(b) 
is sharp apart from an additive constant. This is shown by constructing $3$-connected
quadrangulations in a way similar to 
Example \ref{exa:quadrangulation}. We omit the details.  

We conclude this section with a bound for maximal outerplanar graphs. 

\begin{theorem} \label{theo:rad-vs-pi-in-MOPS}
Let $G$ be a maximal outerplanar graph of order $n$. Then
  \[
{\rm rad}(G) - \pi(G) \leq \frac{n+7}{8}+\frac{2}{n-1}.
\]

\end{theorem}
  
\begin{proof}
The proof follows the lines of the proof of Theorem \ref{theo:rad-vs-pi-in-triangulation}, 
so we give only the main steps. 
Let $v_0$, $N_i(v_0)$, $A_i$, $R$ and $r$ 
be as in the proof of Theorem \ref{theo:rad-vs-pi-in-triangulation}. 

Since every maximal outerplnar graph is $2$-connected, we have $|N_i(v_0)| \geq 2$ for 
$i=1,2,\ldots,R-1$. 
Similar to \eqref{eq:lb on v_0} we prove that 
\begin{equation}  \label{eq:lower-bound-on-sigma(v_0)-in-MOPS} 
\sigma(v_0) \geq 2r^2 - 4r, 
\end{equation}
by considering the two cases $R \geq \frac{1}{2}(3r-3)$ and $R \leq \frac{1}{2}(3r-4)$. 
In the proof of the latter case, we make use of Lemma \ref{lemm:2-connected outerplanar} 
and show, with the same arguments, a slightly altered version of Claim 1 in 
Theorem \ref{theo:rad-vs-pi-in-triangulation}, which states that $|A_i| \geq 4$ holds
for all $i$ with $R-r+2 \leq i \leq 2r-R-2$. \\
It follows from  \eqref{eq:lower-bound-on-sigma(v_0)-in-MOPS} that
the difference ${\rm rad}(G) - \pi(G)$ is bounded from above by $r - \frac{1}{n-1}(2r^2 - 4r)$,
which is maximised for $r=\frac{n+3}{4}$. Substituting this value yields 
Theorem \ref{theo:rad-vs-pi-in-MOPS}. 
\end{proof}

The bound in Theorem \ref{theo:rad-vs-pi-in-MOPS} is sharp apart from an additive constant,
as demonstrated by the following example. 
   
\begin{example}\label{exa:MOP}
Let $n \geq 4$ with $n\equiv 0 \pmod{4}$ and let $k = \frac{n-2}{2}$. We define the graph $MOP_n$ as follows. 
Let $G_0$ be a single-vertex graph with vertex $b_0$. 
For each $i \in \{1, 2, \ldots, k\}$ let $G_i$ be a copy of the path $P_2$ with vertices $a_i$ and $b_i$. 
Let $G_{k+1}$ be a single-vertex graph with vertex $b_{k+1}$. 
The graph $MOP_n$ is obtained from the disjoint union of $G_0, G_1, \ldots, G_k, G_{k+1}$ by adding edges as follows: 
vertex $b_0$ is connected to $a_1$ and $b_1$, 
for each $i \in \{1, 2, \ldots, k-1\}$ the edges $a_i a_{i+1}$ and $b_i b_{i+1}$ are added, 
for each $i \in \left\{1, 2, \ldots, \frac{k-1}{2} \right\}$ the edge $a_i b_{i+1}$ is added, 
for each $i \in \left\{\frac{k+1}{2}, \ldots, k-1 \right\}$ the edge $b_i a_{i+1}$ is added, 
and finally, the edges $a_k b_{k+1}$ and $b_k b_{k+1}$ are added. 
See Figure \ref{fig:MOP} below for a sketch.) \\
 
Clearly, $MOP_n$ is a maximal outerplanar graph of order $n=2k+2$. It is easy to verify that
${\rm rad}(MOP_n) = \frac{n}{4}$ and 
$\pi(MOP_n) = \overline{\sigma}(b_{(k+1)/2}) = \frac{n+1}{8} + \frac{9}{8(n-1)}$.
Hence ${\rm rad}(MOP_n) - \pi(MOP_n) = \frac{n-1}{8} - \frac{9}{8(n-1)}$, which differs from
the bound in Theorem \ref{theo:rad-vs-pi-in-MOPS} by $1 +o(n)$.
For later use  we note that 
${\rm diam}(MOP_n) = \frac{n}{2}$ and 
$\rho(MOP_n) = \overline{\sigma}(b_0) = \frac{n+1}{4} + \frac{1}{4(n-1)}$.
\end{example}

  \begin{figure}[H]
\begin{center}
\resizebox{\linewidth}{!}{
\begin{tikzpicture}
[inner sep=1.5mm,
vertex/.style={circle,thick,draw,minimum size=10pt, font=\Large},
thickedge/.style={line width=1pt}]

% Nodes with labels inside
\node[vertex, fill=green] (b0) at (-10, 2) {$b_0$};
\node[vertex, fill=white] (b6) at (8, 2) {$b_{k+1}$};

\node[vertex, fill=white] (a1) at (-7, 4) {$a_1$};
\node[vertex, fill=white] (b1) at (-7, 0) {$b_1$};

\node[vertex, fill=white] (a2) at (-4, 4) {$a_2$};
\node[vertex, fill=white] (b2) at (-4, 0) {$b_2$};

\node[vertex, fill=white] (a3) at (-1, 4) {$a_{\frac{k+1}{2}}$};
\node[vertex, fill=green] (b3) at (-1, 0) {$b_{\frac{k+1}{2}}$};

\node[vertex, fill=white] (a4) at (2, 4) {$a_{k-1}$};
\node[vertex, fill=white] (b4) at (2, 0) {$b_{k-1}$};

\node[vertex, fill=white] (a5) at (5, 4) {$a_k$};
\node[vertex, fill=white] (b5) at (5, 0) {$b_k$};

% Edges inside each G_i (P_2 paths)
\foreach \i in {1,...,5} {
  \draw[thickedge] (a\i) -- (b\i);
}

% Edge b0 to a1 and b1
\draw[thickedge] (b0) -- (a1);
\draw[thickedge] (b0) -- (b1);

% Edges a_i to a_{i+1} and b_i to b_{i+1}
\foreach \i/\j in {1/2, 4/5} {
  \draw[thickedge] (a\i) -- (a\j);
  \draw[thickedge] (b\i) -- (b\j);
}

\foreach \i/\j in {2/3, 3/4} {
  \draw[thickedge,dashed] (a\i) -- (a\j);
  \draw[thickedge,dashed] (b\i) -- (b\j);
}

% Edges a_i b_{i+1} for i=1 to (k-1)/2 = 2 (solid lines)
\foreach \i/\j in {1/2} {
  \draw[thickedge] (a\i) -- (b\j);
}

\foreach \i/\j in {2/3} {
  \draw[thickedge,dashed] (a\i) -- (b\j);
}

% Edges b_i a_{i+1} for i=(k+1)/2 to k-1 = 3 to 4 (dashed lines)
\foreach \i/\j in {3/4, 4/5} {
  \draw[thickedge, dashed] (b\i) -- (a\j);
}

\foreach \i/\j in {4/5} {
  \draw[thickedge] (b\i) -- (a\j);
}

% Final connections to b_{k+1}
\draw[thickedge] (a5) -- (b6);
\draw[thickedge] (b5) -- (b6);

\end{tikzpicture}
}
\caption{The maximal outerplanar graph $MOP_n$.}
\label{fig:MOP}
\end{center}
\end{figure}
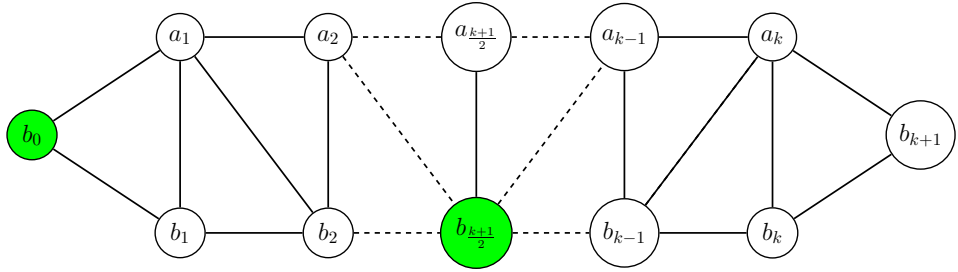

\section{Remoteness vs proximity}
\label{section:remoteness-vs-proximity}

In this section, we prove upper bounds on the difference between remoteness and proximity in graphs.
Our bounds hold not only for maximal planar graphs of given connectivity, but for (not necessarily planar)
graphs of given order and connectivity. 
We also give examples to demonstrate that our bounds are sharp, even if restricted to maximal planar graphs,
quadrangulations or maximal outerplanar graphs.

\begin{theorem}\label{theo:remoteness-proximity}
Let $n, \kappa \in \mathbb{N}$ with $\kappa \leq \frac{n+1}{2}$. Let $G$ be a graph of order 
$n$ with $\kappa(G) = \kappa$. Then
\[
\rho(G) - \pi(G) \leq \frac{n+2\kappa-3}{4\kappa} 
              - \frac{(3\kappa+1)(\kappa-1)}{4\kappa (n-1)}.       \]                           
\end{theorem}

\begin{proof}  
Let \( u \) and \( v \) be vertices of \( G \) such that \( \overline{\sigma}(u) = \rho(G) \) 
and \( \overline{\sigma}(v) = \pi(G) \), and let $k:= d(u,v)$. For $i \in \mathbb{N} \cup \{0\}$ let 
\( N_i(u) \) denote the set of vertices at distance \( i \) from \( u \). 
Our first objective is to find a set \( X \subseteq V(G) \) such that \( \sigma(u|X) - \sigma(v|X) \) 
is small. We then bound \( \sigma(u|V(G) - X) - \sigma(v|V(G) - X) \).

Since $G$ is $\kappa$-connected, we have $|N_i(u)| \geq \kappa$ for all 
$i\in \{1,2,\ldots,{\rm ecc}(u)-1\}$. For \( i \in \{1, 2, \ldots, k-1\} \), define $M_i$ 
to be a set of exactly $\kappa$ vertices of $N_i(u)$, and let \( M_0 =\{u\} \) and \(M_k = \{ v\} \).

Clearly, $M_i \cap M_j = \emptyset$ for all $i, j \in \{0, 1,\ldots,k\}$ with $i\neq j$. 
For \( i \in \{0, 1, \ldots, \lfloor\frac{k}{2}\rfloor\} \) define the sets 
\( B_i = M_i \cup M_{k-i} \) . Let 
\( B = \bigcup_{i=0}^{\lfloor\frac{k}{2}\rfloor} B_i \), so \( |B| = \kappa(k - 1)+2 \).

Note that for each \( y \in M_i \), we have \( d(u, y) = i \) and \( k-i \leq d(v, y) \), and for each \( y \in M_{k-i} \), we have \( d(u, y) = k - i \) and \( i \leq d(v, y)  \). Hence, for $i \neq 0,k$, 
\begin{eqnarray*}
\sigma(u|B_i) - \sigma(v|B_i) 
  &=&  \sigma(u|M_i) + \sigma(u|M_{k-i}) - \sigma(v|M_i) - \sigma(v|M_{k-i}) \\
& \leq&  \kappa i + \kappa(k-i) - \kappa(k-i) - \kappa i \\
& =&  0.
\end{eqnarray*}
By summing the above equation for \( i \in \{1, \ldots, \lfloor\frac{k}{2}\rfloor\} \), and noting 
that \(\sigma(u|B_0) = d(u,v) = \sigma(v|B_0) \), we obtain that 
\begin{equation}\label{eq:total-distance-of-B}
\sigma(u|B) - \sigma(v|B) \leq 0.
\end{equation}
Let $B'=V(G)-B$. Now consider \( w \in B' \). We claim that
\begin{equation} \label{eq:vertices-notin-B} 
d(u,w) - d(v,w) \leq \left\{ \begin{array}{cl}
k-1 & \textrm{if $w \in B' \cap N_{\leq k}(u)$,} \\
k & \textrm{if $w \in B' \cap N_{\geq k+1}(u)$.}
\end{array} \right. 
\end{equation}
Indeed, if $w \in B' \cap N_{\leq k}(u)$, then $d(u,w) \leq k$ and $d(v,w) \geq 1$, so 
$d(u,w) - d(v,w) \leq k-1$. If $w \in B' \cap N_{\geq k+1}(u)$, then 
by the triangle inequality, we have
\[
d(u, w) \leq d(u, v) + d(v, w),
\]
and \eqref{eq:vertices-notin-B} follows since $d(u,v)=k$.

We first show that the theorem holds if $n \leq k\, \kappa$, that is, if there are only few vertices
besides those in $B$. In this case we have  
${\rm diam}(G) \leq 1 + \frac{n-2}{\kappa} \leq 1 + \frac{k\kappa-2}{\kappa} < 1+k$,
and so all vertices of $B'$ are in $B' \cap N_{\leq k}(u)$. Hence by 
\eqref{eq:total-distance-of-B} and \eqref{eq:vertices-notin-B} we obtain
\[ \sigma(u) - \sigma(v) = \sigma(u|B) - \sigma(v|B) + \sum_{w \in B'} \big( d(u,w)-d(v,w)\big)  
           \leq |B'|(k-1). \]
Now $|B'|=n-((k-1)\kappa+2) \leq \kappa-2$ and $n \geq (k-1)\kappa+2$, which implies that
$k-1 \leq \frac{n-2}{\kappa}$. Substituting these bounds into the above inequalities and 
dividing by $n-1$ implies that
\[ \rho(G) - \pi(G) \leq \frac{\kappa-2}{\kappa} \frac{n-2}{n-1}, \]
which by a simple calculation,
yields the theorem. Hence we may assume from now on that $n \geq k \kappa$. 

Since $n > k \kappa$, we have $|B'| \geq \kappa-1$.
We claim that the set $B'$ has at least $\kappa-1$ vertices in $N_{\leq k}(u)$. 
Indeed, if less than $\kappa-1$ vertices of $B'$ are in $N_{\leq k}(u)$, then 
$N_{\geq k+1}$ is nonempty, and the set $N_k(u)$ is a cut-set containing not more than $\kappa-2$ 
vertex besides $v$, a contradiction to $G$ being $\kappa$-connected. 
Summing \eqref{eq:vertices-notin-B} over all \( w \in B' \), we thus get
\begin{eqnarray} 
\sigma(u|B') - \sigma(v|B')  
 & \leq & \sum_{w \in B' \cap N_{\leq k}} (k-1) + \sum_{w \in B' \cap N_{\geq k+1}} k \nonumber \\
 & \leq & k(n-|B|)-(\kappa-1) \nonumber \\ 
 & = & k(n - (k-1)\kappa-2)-\kappa+1. \label{eq:total-distance-of-B'}
\end{eqnarray}
Combining equations (\ref{eq:total-distance-of-B}) and (\ref{eq:total-distance-of-B'}), we obtain that
\begin{eqnarray}
\sigma(u) - \sigma(v) 
& = & \sigma(u|B) - \sigma(v|B) +\sigma(u|B') - \sigma(v|B') \nonumber  \\
& \leq &   k(n - (k-1)\kappa-2)-\kappa+1. \label{eq:rho-pi-1}
\end{eqnarray}
A straightforward maximisation shows that for \( k = \frac{n+\kappa-2}{2\kappa} \), the right-hand side of the above 
inequality is maximised. Substituting this value of \( k \) into inequality 
\eqref{eq:rho-pi-1}, we obtain
\[
\sigma(u) - \sigma(v) \leq \Big( \frac{n+\kappa-2}{2\kappa} \Big)^2 - \kappa + 1.
\]
Dividing by $n-1$ in the above inequality yields, after simplification,
\[
\rho(G) - \pi(G) \leq \frac{n+2\kappa-3}{4\kappa}   
                 - \frac{(\kappa-1)(3\kappa+1)}{4\kappa (n-1)}.
\]
This completes the proof of the theorem.
\end{proof}

The following example shows that for every $\kappa$ the bound in 
Theorem \ref{theo:remoteness-proximity} is sharp whenever $n \equiv \kappa+2 \pmod{2\kappa}$.

\begin{example} \label{ex:rho-pi-connectivity}
Fix $\kappa \in \mathbb{N}$ and let $n \in \mathbb{N}$ with $n \equiv \kappa+2 \pmod{2\kappa}$. 
Let $\ell = \frac{n-2}{\kappa}$. Consider the graph $G_{n, \kappa}$, defined by 
\[ G_{n,\kappa} = K_1 + [K_{\kappa}]^{\ell}  + K_1. \]
Clearly, $G_{n,\kappa}$ is a $\kappa$-connected graph of order $n=\ell \kappa+2$. 
Let $u$ be the vertex in the leftmost $K_1$, and let $v$ be a vertex in the $\frac{\ell+1}{2}$th 
copy of $K_{\kappa}$. It is easy to verify that  
$\pi(G_{n,\kappa}) =\overline{\sigma}(v) = \frac{n+1}{4\kappa} + \frac{3(\kappa^2-1)}{4\kappa(n-1)}$.
$\rho(G_{n,\kappa}) =\overline{\sigma}(u) = \frac{n+\kappa-1}{2\kappa} + \frac{\kappa-1}{2\kappa(n-1)}$.
Hence 
$\rho(G_{n,\kappa}) - \pi(G_{n,\kappa}) 
    = \frac{n+2\kappa-3}{4\kappa} 
              - \frac{(3\kappa+1)(\kappa-1)}{4\kappa (n-1)}$,  
so the bound in Theorem \ref{theo:remoteness-proximity} holds with equality. 
\\
For later use  we note that ${\rm diam}(G_{n,\kappa})=\frac{n+\kappa-2}{\kappa}$. 
\end{example} 

For bipartite graphs, the bound in Theorem \ref{theo:remoteness-proximity} can be improved slightly. 
This improved bound will yield a sharp bound for quadrangulations.

\begin{theorem}\label{theo:remoteness-proximity-in-bipartite-graphs}
Let $n, \kappa \in \mathbb{N}$ with $\kappa \leq \frac{n+1}{2}$. Let $G$ be a bipartite graph of order 
$n$ with $\kappa(G) = \kappa$. Then
\[
\rho(G) - \pi(G) \leq \frac{n+2\kappa-3}{4\kappa}  
               - \frac{(\kappa-1)(7\kappa+1)}{4\kappa (n-1)}.       \]     
\end{theorem}

\begin{proof}

The proof is almost identical to that of Theorem \ref{theo:remoteness-proximity},
with the only difference being that we need the
following slightly stronger version of \eqref{eq:vertices-notin-B}:
\begin{equation} \label{eq:vertices-notin-B-bipartite} 
d(u,w) - d(v,w) \leq \left\{ \begin{array}{cl}
k-2 & \textrm{if $w \in B' \cap N_{\leq k}(u)$,} \\
k & \textrm{if $w \in B' \cap N_{\geq k+1}(u)$.}
\end{array} \right. 
\end{equation}
Indeed, if $w \in B' \cap N_{k}(u)-\{w\}$, then $d(u,w) = k$ and $d(v,w) \geq 2$, 
and if $w \in B' \cap N_{\leq k-1}(u)-\{w\}$, then $d(u,w) \leq k-1$ and $d(v,w) \geq 1$,
so $d(u,w) - d(v,w) \leq k-2$ in either case. If $w \in B' \cap N_{\geq k+1}(v)$, then 
$d(u,w) - d(v,w) \leq k$ by the triangle inequality as in the proof of \eqref{eq:vertices-notin-B}.

The remainder of the proof is identical to that of Theorem \ref{theo:remoteness-proximity}, 
and thus omitted.
\end{proof}

The sharpness of Theorem \ref{theo:remoteness-proximity-in-bipartite-graphs} is shown by the following
example. 

\begin{example}   \label{ex:rho-pi-connectivity-bipartite}
Fix $\kappa \in \mathbb{N}$ and let $n \in \mathbb{N}$ with $n \equiv \kappa+2 \pmod{2\kappa}$. 
Let $\ell = \frac{n-2}{\kappa}$. Consider the graph $\overline{G}_{n, \kappa}$, defined by 
\[ \overline{G}_{n,\kappa} = K_1 + [\overline{K_{\kappa}}]^{\ell}  + K_1. \]
Clearly, $\overline{G}_{n,\kappa}$ is a bipartite, $\kappa$-connected graph of order 
$n=\ell \kappa+2$. 
Let $u$ be the vertex in the leftmost $K_1$, and let $v$ be a vertex in the $\frac{\ell+1}{2}$th 
copy of $\overline{K_{\kappa}}$. It is easy to verify that  
$\pi(\overline{G}_{n,\kappa}) =\overline{\sigma}(v) 
   = \frac{n+1}{4\kappa} + \frac{7\kappa^2-4\kappa-3}{4\kappa (n-1)}$.
$\rho(\overline{G}_{n,\kappa}) =\overline{\sigma}(u) 
      = \frac{n+\kappa-1}{2\kappa} + \frac{\kappa-1}{2\kappa(n-1)}$.
Hence 
$\rho(\overline{G}_{n,\kappa}) - \pi(\overline{G}_{n,\kappa}) 
    = \frac{n+2\kappa-3}{4\kappa}  
               - \frac{(\kappa-1)(7\kappa+1)}{4\kappa (n-1)}$. 
so the bound in Theorem \ref{theo:remoteness-proximity-in-bipartite-graphs} holds with equality. 
\\
For later use  we note that ${\rm diam}(\overline{G}_{n,\kappa})=\frac{n+\kappa-2}{\kappa}$. 
\end{example}

As corollaries to Theorems \ref{theo:remoteness-proximity} and 
\ref{theo:remoteness-proximity-in-bipartite-graphs} we obtain sharp bounds on the 
difference between remoteness and proximity for the classes of planar graphs
considered in Section \ref{section:radius-vs-proximity}.

\begin{corollary} \label{coro:rho-pi-in-planar}
(a) If $G$ is a maximal planar graph of order $n$, then 
\[ \rho(G) - \pi(G) \leq \frac{n+3}{12}  - \frac{5}{3(n-1)}. \] 
(b) If $G$ is a $4$-connected maximal planar graph of order $n$, then 
\[ \rho(G) - \pi(G) \leq \frac{n+5}{16}  - \frac{55}{16(n-1)}. \]
(c) If $G$ is a $5$-connected maximal planar graph of order $n$, then 
\[ \rho(G) - \pi(G) \leq \frac{n+7}{20}  - \frac{26}{5(n-1)}. \]
(d) If $G$ is a maximal outerplanar graph of order $n$, then
\[ \rho(G) - \pi(G) \leq \frac{n+1}{8}  + \frac{1}{8(n-1)}. \] 
The bounds in (a)-(d) are sharp for infinitely many values of $n$. \\
\end{corollary}

\begin{proof}
Since every maximal planar graph is $3$-connected, and since every maximal outerplanar
graph is $2$-connected, we obtain the bounds in (a) and (d) from
Theorem \ref{theo:remoteness-proximity} for $\kappa=3$ and $\kappa=2$, respectively. \\ 
The maximal planar graphs in Example \ref{exa:triangulation-3-connected} attain the 
bound in (a), so the bound is sharp whenever $n \equiv 5 \pmod{6}$.
The maximal outerplanar graphs in Example  \ref{exa:MOP} attain the 
bound in (d), so the bound is sharp whenever $n \equiv 0 \pmod{4}$. \\
Similar examples show that also the bounds in (b) and (c) are sharp if 
$n \equiv 6\pmod{8}$ and $n \equiv 7 \pmod{10}$, respectively.   
\end{proof}

Now we consider quadrangulations. Since every quadrangulation is bipartite and $2$-connected,
we obtain the following corollary to Theorem \ref{theo:remoteness-proximity-in-bipartite-graphs}.
The quadrangulations in Example \ref{exa:MOP} show that the bound in (a) is sharp 
for $n \equiv 0 \pmod{4}$.
In a similar way one can construct $3$-connected maximal planar graphs that show that the bound
in (b) is sharp whenever $n \equiv 5 \pmod{6}$.

\begin{corollary}\label{coro:remotness-proximity-in-quadrangulation}
(a)  If $G$ is a quadrangulation of order $n$, then 
  \[
\rho(G) - \pi(G) \leq \frac{n+1}{8}  - \dfrac{15}{8(n-1)}.
\]
(b)  If $G$ is a $3$-connected quadrangulation of order $n$, then 
  \[
\rho(G) - \pi(G) \leq \frac{n+3}{12}  - \frac{11}{3(n-1)}.
\]
\end{corollary}
  
The graph in Example \ref{exa:quadrangulation} shows that the bound in 
Corollary \ref{coro:remotness-proximity-in-quadrangulation}(a) is sharp whenever 
$n \equiv \kappa+2 \pmod{2\kappa}$. 
It is not hard to construct a similar $3$-connected quadrangulation that shows that also
Corollary \ref{coro:remotness-proximity-in-quadrangulation}(b) is sharp whenever 
$n \equiv \kappa+2 \pmod{2\kappa}$.

\section{Diameter vs Proximity}
\label{section:diameter-vs-proximity}

In this section we compare diameter and proximity of graphs of given order and connectivity,
and we show that the bound in Theorem \ref{theo:diameter-vs-pi-given-n} can be strengthened 
for $\kappa$-connected graphs by a factor of about $\frac{1}{\kappa}$. 
We first prove a sharp lower bound on the proximity in terms of order and connectivity, 
which we use to obtain an upper bound on the difference between diameter and proximity.

As corollaries we obtain bounds on the difference between diameter and proximity for maximal planar graphs, 
quadrangulations, and maximal outerplanar graphs. We construct graphs to show that these bounds are sharp
apart from an additive constant. 

In the proof of the main result of this section we make use of the following well-known
bound on the diameter due to Watkins \cite{Wat1967}. 

\begin{proposition} \cite{Wat1967}  \label{prop:diam-kappa-connected} 
If $G$ is a $\kappa$-connected graph of order $n$, then 
\[ {\rm diam}(G) \leq \left\lfloor \frac{n + \kappa - 2}{\kappa} \right\rfloor. \]
\end{proposition}

\begin{theorem}   \label{theo:diameter-vs-p-given-connectivity}
(a) Let $G$ be a $\kappa$-connected graph of order $n$ and diameter $d$. Then
\[
\pi(G) \geq \left\{ \begin{array}{cc} 
\frac{\kappa(d-3)^2}{4(n-1)} + 1 + \frac{4(d-2)-\kappa}{4(n-1)} & \textrm{if $d$ is even,} \\
\frac{\kappa(d-3)^2}{4(n-1)} + 1 + \frac{d-2}{n-1} & \textrm{if $d$ is odd,} 
\end{array} \right. 
\]
(b) Let $G$ be a $\kappa$-connected graph of order $n$. Then 
\[ {\rm diam}(G) - \pi(G) \leq  \frac{3n-9}{4\kappa } + 1  - \frac{3\kappa^2 - 3}{4\kappa(n-1)}. \] 
\end{theorem}

\begin{proof}
Let \( G \) be a $\kappa$-connected graph of order $n$. 
Define $d = \text{diam}(G)$, and let $v, w_0, w_d$ be vertices of $G$, where $\overline{\sigma}(v) = \pi(G)$ 
and $d(w_0, w_d) = \text{diam}(G)$.  
Since $G$ is $\kappa$-connected, each $N_i(w_0)$, $i \in \{1,2,\ldots,d-1\}$, contains at least 
$\kappa$ vertices. For each $i = 1, 2, \ldots, d-1$  define $M_i$ to be a set of $\kappa$ vertices  of $N_i(w_0)$. 
Clearly, $M_i \cap M_j = \emptyset$ for all $i, j \in \{1,2,\ldots,d-1\}$ with $i\neq j$. 
Define $B_0 = \{w_0,w_d\}$ and $B_i = M_i \cup M_{d-i}$ for $i = 1, \ldots, \lfloor \frac{d}{2} \rfloor-1$, 
and let $B = \bigcup_{i=0}^{\lfloor \frac{d}{2} \rfloor -1} B_i$. Then $|B| = 2 + \kappa(d-2)$ if $d$ is even, 
and $|B|=2+\kappa(d-3)$ if $d$ is odd.  

To bound \( \sigma(v|B) \) from below, consider \( \sigma(v|M_i \cup M_{d-i}) \) for \( i = 1, \ldots, \lfloor \frac{d}{2} \rfloor-1 \). Using the triangle inequality, we obtain 
\( d(v, a) + d(v, b) \geq d(a, b) \geq d - 2i \) for any \( a \in M_i \) and \( b \in M_{d-i} \). 
This yields the following bound.
\[
\sigma(v|B_i) = \sigma(v|M_i) + \sigma(v|M_{d-i}) \geq \kappa (d - 2i).
\]
Summing over all \( i \in \{1, \ldots, \lfloor \frac{d}{2} \rfloor -1\} \), and taking 
into account that $d(v,w_0) + d(v,w_d) \geq d(w_0,w_d) = d$, we obtain:
\[
\sigma(v|B) = \sum_{i=0}^{\lfloor \frac{d}{2} \rfloor -1} \sigma(v|B_i) 
  \geq d + \sum_{i=1}^{\lfloor \frac{d}{2} \rfloor -1} \kappa(d - 2i).
\]
Hence, 
\begin{equation}\label{transmission of B}
\sigma(v|B) \geq \left\{ \begin{array}{cc}
d+ \frac{1}{4} \kappa (d^2-2d) & \text{if $d$ is even,} \\
d+ \frac{1}{4}\kappa(d^2-2d-3) & \text{if $d$ is odd.}
\end{array} \right.
\end{equation}
We now bound the transmission of $v$ from below. Since there are $n-|B|$ vertices not in $B$,
and of those at least $n-|B|-1$ are distinct from $v$, we obtain by \eqref{transmission of B} that
\begin{eqnarray}
\sigma(v, G) & \geq &  \sigma(v|B) + (n-|B|-1) \nonumber \\
& \geq & \left\{ \begin{array}{cc}
d+ \frac{1}{4} \kappa (d^2-2d) + n - \kappa(d-2) -3 & \text{if $d$ is even,} \\
d+ \frac{1}{4}\kappa(d^2-2d-3) +n - \kappa(d-3) - 3 & \text{if $d$ is odd.}
\end{array} \right. \label{transmission of v in terms of d}
\end{eqnarray}
Dividing by \( n-1 \) in (\ref{transmission of v in terms of d}) yields, after simplification, 
\[
\pi(G) \geq \left\{ \begin{array}{cc} 
\frac{\kappa(d-3)^2}{4(n-1)} + 1 + \frac{d-2}{n-1} - \frac{\kappa}{4(n-1)}& \textrm{if $d$ is even,} \\
\frac{\kappa(d-3)^2}{4(n-1)} + 1 + \frac{d-2}{n-1} & \textrm{if $d$ is odd,} 
\end{array} \right. 
\]
which is the desired lower bound on $\pi(G)$. This proves (a). \\[2mm]
(b) We now bound the difference between the diameter and proximity. Recall that $d={\rm diam}(G)$. 
From (a) we have that
$\pi(G) \geq \frac{\kappa(d-3)^2}{4(n-1)} + 1 + \frac{d-2}{n-1}- \frac{\kappa}{4(n-1)}$. 
Hence
\begin{equation} \label{eq:d-vs-pi} 
{\rm diam}(G) - \pi(G) \leq d - \frac{\kappa(d-3)^2}{4(n-1)} - 1 - \frac{d-2}{n-1} + \frac{\kappa}{4(n-1)}. 
\end{equation}
Denote the right-hand side of \eqref{eq:d-vs-pi} by \( f(d) \). 
Then $f'(d)= 1 -  \frac{1}{2(n-1)}[\kappa(d-3)+2]$. Since $d \leq \frac{n+\kappa-2}{\kappa}$ by 
Proposition \ref{prop:diam-kappa-connected}, it follows that $f'(d) >0$, so $f$ is increasing and 
$f(d)$ attains its maximum for $1 \leq d \leq \frac{n+\kappa-2}{\kappa}$ at $d= \frac{n+\kappa-2}{\kappa}$.  
Substituting this value into \eqref{eq:d-vs-pi}  yields, after simplification, 
\[
{\rm diam}(G) - \pi(G)  \leq   f( \frac{n+\kappa-2}{\kappa}) 
   =   \frac{3n-9}{4\kappa } + 1  - \frac{3\kappa^2 - 3}{4\kappa(n-1)},
\]
as desired. 
\end{proof}

We note that the case $\kappa=1$ yields a lower bound on the proximity of a tree of 
given order $n$ and diameter $d$ given in \cite{PenZho2021} if $n$ is odd. 
It also yields the bound on the difference 
${\rm diam}(G) - \pi(G)$ in Theorem \ref{theo:diameter-vs-pi-given-n} for 
the case that $n$ is odd. 
For even $n$, our bound for $\kappa=1$ 
differs from that in Theorem \ref{theo:diameter-vs-pi-given-n} by a term $o(n)$. 

The bound in Theorem \ref{theo:diameter-vs-p-given-connectivity}(a) is sharp 
for  every $n.d,\kappa$ with $d \leq \frac{n-2}{\kappa}+1$. 

\begin{example}  \label{ex:diameter-pi-connectivity-1}
Let $n,d,\kappa \in \mathbb{N}$ with $d \leq \frac{n-2}{\kappa}+1$. 

First let $d$ be even. Define $\ell=\frac{d}{2}$, and let 
\[ G = K_1 + [K_{\kappa}]^{\ell-1} + K_{n - \kappa(d-2)-2} + [K_{\kappa}]^{\ell-1} + K_1. \]
Clearly, $G$ is a $\kappa$-connected graph of order $n$ and diameter $d$, and the median vertices 
of $G$ are exactly the vertices in $K_{n - \kappa(d-2)-2}$. Let $v$ be a median vertex.
A straightforward calculation shows that 
\[ \sigma(v) = (n - \kappa(d-2)-3) +  \sum_{i=1}^{\ell-1} 2\kappa i + 2\ell. \]
Dividing by $n-1$ and simplifying yields
\[ \pi(G) = \overline{\sigma}(v) = \frac{\kappa(d-3)^2}{4(n-1)} + 1 + \frac{4(d-2)-\kappa}{4(n-1)}, \]
which is the bound in Theorem \ref{theo:diameter-vs-p-given-connectivity}(a) for even $d$. 

Now let $d$ be odd. Define $\ell=\frac{d-1}{2}$, and let 
\[ G = K_1 + [K_{\kappa}]^{\ell} + K_{n - \kappa(d-2)-2} + [K_{\kappa}]^{\ell-1} + K_1. \]
It is easy to see that the vertices in $K_{n - \kappa(d-2)-2}$ are median vertices of $G$. 
Calculations similar to the case that $d$ is even show that their average distance,
i.e., $\pi(G)$, attains the bound in Theorem \ref{theo:diameter-vs-p-given-connectivity}(a).
\end{example}

Example \ref{ex:rho-pi-connectivity} shows that if, in addition,
$n \equiv \kappa+2 \pmod{2\kappa}$, then the bound in 
Theorem \ref{theo:diameter-vs-p-given-connectivity}(b) is also sharp.

\begin{corollary}\label{coro:diam-vs-proximity}
Let \( G \) be a planar graph of order \( n \geq 4 \).  \\[1mm]
(a) If $G$ is a maximal planar graph, then
\[
\ {\rm diam}(G)- \pi(G) \leq \frac{n+1}{4}-\frac{2}{n-1}.
\]
(b) If $G$ is a 4-connected maximal planar graph, then
\[
{\rm diam}(G)- \pi(G) \leq \frac{3n+7}{16} - \frac{45}{16(n-1)}.
\]
(c) If $G$ is a 5-connected maximal planar graph, then
\[
{\rm diam}(G) - \pi(G) \leq \frac{3n+11}{20} - \frac{18}{5(n-1)}.
\]
(d) If $G$ is a quadrangulation, then
\[
{\rm diam}(G) - \pi(G) \leq \frac{3n-1}{8} - \frac{9}{8(n-1)}.
\]
(e) If $G$ is a $3$-connected quadrangulation, then
\[
{\rm diam}(G) - \pi(G) \leq \frac{n+1}{4} - \frac{2}{n-1}.
\]
(f) If $G$ is a maximal outerplanar graph, then
\[
{\rm diam}(G) - \pi(G) \leq \frac{3n-1}{8} + \frac{9}{8(n-1)}.
\]
\end{corollary}

The graphs defined in Example \ref{exa:triangulation-3-connected} show that the bound in
Corollary \ref{coro:diam-vs-proximity} is sharp. In a similar way, it is easy to construct
$4$-connected and $5$-connected planar graphs that demonstrate the sharpness of the bounds 
in (b) and (c).

The graphs in Example \ref{exa:quadrangulation} show that the bound in (d) is sharp part from an additive constant. 
In a similar way, it is easy to construct $3$-connected quadrangulations that demonstrate 
the sharpness, up to an additive constant, of the bound in (e).

Finally, the graphs in Example \ref{exa:MOP} show that the bound in (f) is sharp.

\end{document}